\documentclass[a4paper,12pt]{amsart}

\usepackage[all]{xy}

\usepackage{t1enc}
\def\cD{\mathcal{D}}
\def\cR{\mathcal{R}}
\def\cV{\mathcal{V}}

\def\cN{\mathcal{N}}
\def\cU{\mathcal{U}}
\def\cL{\mathcal{L}}

\def\R{\mathbb{R}}

\def\F{\mathbb{F}}

\def\be{\beta}
\def\al{\alpha}
\def\la{\lambda}

\def\bx{\mbox{\boldmath $x$}}
\def\bD{\mbox{\boldmath $\cD$}}
\def\bg{\mathbf{g}}

\newcommand{\id}{\operatorname{id}}

\newcommand{\End}{\operatorname{End}}
\newcommand{\Min}{\operatorname{Min}}
\newcommand{\Max}{\operatorname{Max}}

\newtheorem{theorem}{Theorem}[section]
\newtheorem{lemma}[theorem]{Lemma}
\newtheorem{proposition}[theorem]{Proposition}

\theoremstyle{remark}
\newtheorem{remark}[theorem]{\rm\bf Remark}
\newtheorem*{definition*}{\rm\bf Definition}

\newcommand{\nn}[1]{(\ref{#1})}

\def\sideremark#1{\ifvmode\leavevmode\fi\vadjust{\vbox to0pt{\vss
 \hbox to 0pt{\hskip\hsize\hskip1em
 \vbox{\hsize3cm\tiny\raggedright\pretolerance10000
  \noindent #1\hfill}\hss}\vbox to8pt{\vfil}\vss}}}%
                        
\def\idx#1{{\em #1\/}}

\author{A. Rod Gover and Josef \v Silhan}
\email{gover@math.auckland.ac.nz} \title{Commuting linear operators and algebraic decompositions} 
\begin{document}

\address{ARG: Department of Mathematics\\
  The University of Auckland\\
  Private Bag 92019\\
  Auckland 1\\
  New Zealand} \email{gover@math.auckland.ac.nz}
\address{JS: Eduard \v{C}ech center \\ 
Department of Algebra and geometry \\
Masaryk University \\
Jan\'a\v{c}kovo n\'am. 2a \\
602 00, Brno\\
Czech Republic} \email{silhan@math.muni.cz}

\maketitle

\pagestyle{myheadings} \markboth{Gover \& \v Silhan}{Commuting linear
operators and decompositions}

\begin{abstract}
  For commuting linear operators $P_0,P_1,\cdots ,P_\ell$ we describe
  a range of conditions which are weaker than invertibility. When any
  of these conditions hold we may study the composition
  $P=P_0P_1\cdots P_\ell$ in terms of the component operators or
  combinations thereof.  In particular the general inhomogeneous
  problem $Pu=f$ reduces to a system of simpler problems. These
  problems capture the structure of the solution and range spaces and, if the
  operators involved are differential, then this gives an effective way
  of lowering the differential order of the problem to be studied.
  Suitable systems of operators may be treated analogously. For a
  class of decompositions the higher symmetries of a composition $P$
  may be derived from generalised symmmetries of the component
  operators $P_i$ in the system.
\end{abstract}

\section{Introduction}
\newcommand{\bF}{\mathbb{F}} \newcommand{\bC}{\mathbb{C}}
\newcommand{\Proj}{\operatorname{Proj}}
\renewcommand{\Pr}{\operatorname{Pr}}

Given a vector space $\cV$ and a system $P_0,P_1,\cdots ,P_\ell$ of
mutually commuting endomorphims of $\cV$ we study the composition $P:=
P_0P_1\cdots P_\ell$. It is natural to ask whether we can reduce the
questions of null space and range of $P$ to the similar questions for
the component operators $P_i$. If these component operators are each
invertible then of course one trivially has a positive answer to this
question.  On the other hand experience with, for example, constant
coefficient linear ordinary differential equations shows that this is
too much to hope for in general. Here we review, discuss, and extend a
recent work \cite{GoSiDec} in which we introduce a range of conditions
which are significantly weaker than invertibility of $P$ and yet
which, in each case, enables  progress along these lines. 

Each condition we describe on the system $P_0,P_1,\cdots ,P_\ell$ is
termed an $\al$-decomposition (where $\al$ is a subset of the power
set of the index set $\{0,\cdots ,\ell\}$). The case that the
operators $P_i$ are each invertible is one extreme. Of course one may
ask that some of $P_i$ are invertible but, excluding an explicit
assumption along these lines, the next level is what we term as simply
a decomposition. This is described explicitly in Section \ref{setup}
below, but intuitively the main point is that each pair $P_i$ and
$P_j$, for $i\neq j$ in $\{0,\cdots ,\ell\}$, consistes of operators
which are relatively invertible in the sense that for example $P_i$ is
invertible on the null space of $P_j$ and vice versa. In the case that
we have a decomposition then one obtains very strong results: the null
space of $P$ is exactly the direct sum of the null spaces for the
factors $P_i$; the range of $P$ is precisely the intersection of the
range spaces for the factors; and one may explicitly decompose the
general inhomogeneous problem $P u=f $ into an equivalent system of
``lower order'' problems $P_i u_i=f$, $i=0,\cdots ,\ell$. (Note that
the fact the same inhomogeneous term $f\in \cV$ appears in $P u=f$ and
in each of the $P_i u_i=f$ problems is one signal that the
construction we discuss is not the trivial manouever of renaming
variables.)

At the other extreme of the $\al$-decompositions we ask only that the
operator $(P_0,\cdots ,P_\ell):\cV\to \oplus^{\ell+1}\cV$ is injective
with left inverse given by a system of $\cV$ endomorphisms $Q_i$,
$i=0,\cdots ,\ell$, which commute with the $P_j$s. Remarkably this is
sufficient for obtaining results along a similar line to the case of a
decomposition, but the extent of simplification is less drammatic:
issues of null space and range for $P$ are subordinated only to
similar questions for the operators $P^i:= P/P_i$.  The full summary
result for $\al$-decompositions is given in Theorem \ref{winhg}
 below.

A natural setting for the use of these results is in the study of
operators $P$ which are polynomial in a mutually commuting system of
linear operators $\cD_j:\cV\to \cV$, $j=1,\cdots, k$. This is the
subject of Section \ref{multivar}. Given a $P$ which is suitably
factored, or alternatively working over an algebraically closed field,
one sees that generically some algebraic $\al$-decomposition is
available. The main point here is the word ``algebraic''. The
$\al$-decompositions of compositions $P=P_0P_1\cdots P_\ell$ involve
identities which involve operators $Q_i$, $i=0,\cdots ,\ell$, which
invert some subsystem of $(P_0,\cdots ,P_\ell)$. In the case of an
algebraic decomposition the $Q_i$s are also polynomial in the
$\cD_j$s. So, for example, if the factors $P_i$ of $P$ are
differential and polynomial in the differential operators $\cD_i$ then
the $Q_i$ are also differential and are given in terms of the $\cD_i$
by explicit algebraic formulae. In particular pseudo-differential
calculus is avoided. These results are universal in the sense that they are 
independent of any details of the operators $\cD_j$.

The idea behind the decompositions of the equation $P_0 \cdots P_\ell
u=f$ is rather universal; for example one can extend it to systems of
equations.  In this direction our aim is mainly to demonstrate the
technique, so we shall treat, as an example (see Section
\ref{systems}), only one specific situation where the idea applies.
This is a system of $k+1$ equations where the first equation $P_0
\cdots P_\ell u=f$ is factored and the remaining ones are of the form
$R^{(j)}u = g^j$, $1 \leq j \leq k$ for given $f, g^j \in \cV$ and
$R^{(j)}\in \End (\cV)$. This may be viewed as a problem where one
wants to solve the problem $P_0 \cdots P_\ell u=f$, subject to
the conditions $R^{(j)}u = g^j$. The difference, in comparison to the
single equation problem, is that now the operators $R^{(j)}$ feature
in the relative invertibility of the factors $P_i$ of $P$. The result
is that provided one has a suitable decomposition at hand, the
original system is equivalent to a family of ``lower order'' systems
of the same type as the original one.

Finally in Section \ref{symm} we discuss symmmetries of operators.  We
 define a formal symmetry of an operator $P:\cV\to \cV$ to be an
 operator $S:\cV\to \cV$ such that $PS=S'P$ for some other operator
 $S'$ on $\cV$. For $P$ a Laplacian (or Laplacian power) type operator
 differential operator, and $S$, $S'$ differential, such symmetries are
 central in the separation of variable techniques \cite{BKM,Mi}. For
 $P=P_0P_1\cdots P_\ell$, as above, the tools we develop earlier are
 used to show that the formal symmetry algebra of $P$ is generated by
 the formal symmetry operators, and appropriate generalisations
 thereof, for the component operators $P_0P_1, \cdots ,P_\ell$.

The first author would like to thank the Royal Society of New Zealand
for support via Marsden Grant no. 06-UOA-029.
The second author was supported from the Basic Research Center no.\ LC505
(Eduard \v{C}ech Center for Algebra and Geometry) of Ministry of Education,
Youth and Sport of Czech Republic.

\section{Decompositions and $\alpha$-decompositions} \label{setup}

Let $\cV$ denote a vector space over a field $\bF$ and consider linear
operators $P_0, \ldots,P_\ell$ (i.e. endomorphisms of $\cV$) which
mutually commute.  In \cite{GoSiDec} we study properties of the
operator
\begin{equation}\label{compP}
P=P_0P_1\cdots P_\ell.
\end{equation}
For example, an obvious question is: what can we say about the kernel
and the image of $P$ in terms of related data for the component
operators $P_i$?  This is clearly straightforward if the operators
$P_i$ are invertible, but the point of our studies is that much weaker
assumptions are sufficient to obtain quite striking results. These
assumptions are captured in the notion of various ``decompositions'';
the different possible decompositions are parametrised by a nonempty
system of subsets of $L := \{0,\ldots,\ell\}$.  We shall use the
notation $P_J := \prod_{j \in J} P_j$ for $\emptyset \not= J \subseteq
L$ and we set $P_\emptyset := id_\cV$.  Further we put $P^J:= P_{L
\setminus J}$, $P^j := P^{\{j\}}$ and write $2^L$ to denote the power
set of $L$. Also $|J|$ will denote the cardinality of a set $J$.

\begin{definition*} For a linear operator $P:\cV \to \cV$, an
expression of the form \nn{compP} will be said to be an
{\em $\al$--decomposition} of $P$, with 
$\emptyset \not= \al \subseteq 2^L$, $L \not\in \al$, 
if there exist operators $Q_J \in \End(\cV)$, $J \in \al$ such that
\begin{equation} \label{al-iddec}
  \id_{\cV} = \sum_{J \in \al} Q_J P^J, \quad
  [P_i,P_k] = [Q_J,P_i] =0, \quad
  i,k \in L, J \in \al.
\end{equation}
\end{definition*}

The choice $\al := \{\{\emptyset\}\}$ means that $P$ (hence also each
of the $P_i$) is invertible. The other possible decompositions involve
weaker assumptions on the component operators. At the next level is
the $\al$--decomposition with $\al := \{ J \subseteq L \mid |J|=1 \}$
which will be termed simply a \idx{decomposition} of $P$. In this case
we still obtain, for example, that $\cN(P) = \bigoplus
\cN(P_i)$. Therefore the problem $Pu=0$, $u \in \cV$ is reduced to the
system $P_iu_i=0$, $u_i \in \cV$ for $i \in L$. In the case that the
$P_i$ are differential operators, this result shows that, given a
decomposition, the equation $Pu=0$ reduces  to the lower order system
$P_iu_i=0$.

For the general $\al$--decomposition we do not 
generally obtain a direct sum analogous to $\cN(P) = \bigoplus \cN(P_i)$
as above, however
we still get a reduction to a ``lower order'' problem. 
The key is the following theorem which is 
a central result in \cite{GoSiDec}. (See the latter for the proof 
and more details.)

\begin{theorem}\cite{GoSiDec} \label{winhg}
Assume $P:\cV\to \cV$ as in \nn{compP} is an $\al$--decomposition.
Let us fix $f\in \cV$.
There is a surjective mapping $B$ from the space of solutions
$(u_J)_{J \in \al} \in \oplus^{|\al|} \cV$ of the
problem
\begin{equation}\label{wmultiprobg}
 P_J u_J =f, \quad J \in \al.
\end{equation}
onto the space of solutions $u\in \cV$ of $P u=f$.

Writing $\cV_P^f$ for the solution space of $P u=f$ and (for $J \in \al$)
$\cV^f_J$ for the
solution space of $P_J \tilde{u}=f$.  The map
$B: \mbox{\Large\bf $\times$}_{J \in \al}  \cV^f_J \to \cV^f_P$ is given by
$$ (u_J)_{J \in \al} \mapsto \sum_{J \in \al}
   Q_J  u_J ~.
$$
A right inverse for this is
$F:\cV^f_P \to \mbox{\Large\bf $\times$}_{J \in \al} \cV^f_J$  
given (component-wise) by
$$ u\mapsto P^J u  ~;$$
on $\cV$ we have $B\circ F=id_{\cV}$.

If $\al$ satisfies $I\cap J=\emptyset$, for all $I\neq J\in \al$, then
$F$ is a 1-1 mapping and $F\circ B$ is the identity on the solution space to
\nn{wmultiprobg}.
\end{theorem}

\begin{remark} \label{dechg} 1. The important feature of the decomposition of  
inhomogeneous problems is that in $Pu=f$ and \nn{wmultiprobg} it is the
{\bf same} $f\in \cV$ involved. So \nn{wmultiprobg} describes the range
$P$, $\cR(P)$, in terms of the range of the $P_i$:
$\cR(P)=\cap_{i=0}^\ell\cR(P_i)$.

2. The condition in the last paragraph in the theorem is satisfied by a
decomposition, but is easy to construct other examples. In any case
where this is satisfied the mappings $F$ and $B$ are bijections and we
do get a direct sum decomposition of $\cN(P)$. The point, which is
easily verified, is that from \nn{al-iddec} it follows that for each
$J\in\alpha$
$$ 
Q_JP^J:\cN(P)\to \cN(P_J)
$$
is a projection.
\end{remark}

Although the $\al$--decomposition \nn{al-iddec} is what is directly
employed in the previous theorem and its proof, there is a distinct,
but related, notion which shows what it really means for the commuting
operators $P_i$.  The following definition introduces an idea of
a decomposition which turns about to be in a suitable sense ``dual''
to the previous one.
\begin{definition*}
For a linear operator $P:\cV \to \cV$, an
expression of the form \nn{compP} will be said to be a
{\em dual $\beta$--decomposition} of $P$, $\emptyset \not= \be \subseteq 2^L$,
$\{\emptyset\} \not= \be$ if for
every $J \in \be$ there exist operators $Q_{J,j} \in \End(\cV)$, $j \in J$
such that
\begin{equation} \label{iddecdual}
  \id_{\cV} = \sum_{j \in J} Q_{J,j} P_j, \quad
  [P_i,P_k] = [Q_{J,j},P_i] =0, \quad
  i,k \in L, j \in J.
\end{equation}
\end{definition*}

To describe the suggested duality (see Proposition \ref{Peucgen} below), 
first observe that
each system $\al \subseteq 2^L$ is partially ordered by restricting
the poset structure of $2^L$. The sets of minimal and maximal elements
in $\al$ will be denoted by $\Min(\al)$ and $\Max(\al)$, respectively.
We say the system $\be \subseteq 2^L$ is a \idx{lower set}, if it is
closed under taking a subset. (That is, if $I \in \be$ and $J
\subseteq I$ then $J \in \be$.) The \idx{upper set} is defined
dually. The lower set and upper set generated by a system $\al
\subseteq 2^L$ will be denoted by
$\cL(\al) := \{ J\subseteq I~|~ I\in \al\}$
and $\cU(\al) :=  \{ J \supseteq I~|~ J \subseteq L~\mbox{and}~I\in \al\}$,
respectively.

The proof of the following is obvious.

\begin{lemma} \label{MinMax}
Let $\al \subseteq 2^L$. Then
$P = P_0 \cdots P_\ell$ satisfies the following: \\
(i) it is an $\al$--decomposition $\Longleftrightarrow$
it is a $\Max(\al)$--decomposition $\Longleftrightarrow$
it is an $\cL(\al)$--decomposition \\
(ii) it is a dual $\al$--decomposition $\Longleftrightarrow$
it is a dual $\Min(\al)$--decomposition $\Longleftrightarrow$
it is a dual $\cU(\al)$--decomposition.
\end{lemma}

To formulate the relation between $\al$-- and dual
$\al$--decompositions, we need the following notation. We put $\al^u
:= 2^L \setminus \cL(\al)$
and $\al^l := 2^L
\setminus \cU(\al)$. Clearly $(\al^u)^l = \cL(\al)$ and $(\al^l)^u =
\cU(\al)$.
Also it is easily seen that
\begin{eqnarray} \label{UL}
\begin{split}
  &\al^u = \{J \subseteq L \mid \forall I \in \al:
            J \setminus I \not= \emptyset \} \\
  &\al^l = \{J \subseteq L \mid \forall I \in \al:
            I \setminus J \not= \emptyset \}.
\end{split}
\end{eqnarray}

\begin{proposition}[The duality] \label{Peucgen}
\nn{compP} is an $\al$--decomposition if and only if
it is a dual $\al^u$--decomposition. Equivalently, \nn{compP} is a
dual $\be$--decomposition if and only if it is a $\be^l$--decomposition.

In particular, \nn{compP} is a decomposition if and only if
it is a dual $\be$--decomposition for 
$\be := \{ J \subseteq L \mid |J|=2 \}$. This means
\begin{equation} \label{paireucgen}
  \id_{\cV} = Q_{i,j} P_i +Q_{j,i} P_j
\end{equation}
where $Q_{i.j}\in \End (\cV)$ and satisfy $[Q_{i,j},P_k]=0$ for every
triple of integers $(i,j,k)$ such that $0 \leq i , j,k \leq \ell$ and
$i \not= j$.
\end{proposition}

When \nn{paireucgen} is satisfied we shall say that the operators $P_i$
and $P_j$ are \idx{relatively invertible}. The dual version of a
(true) decomposition is the dual $\be$--decomposition for $\be = \{ J
\subseteq L \mid |J|=2 \}$; this will be termed simply a \idx{dual
decomposition}.  The general dual $\be$--decomposition means that for
every $J \in \be$, the operators $P_j$, $j \in J$ are relatively
invertible.

\begin{remark} 
1.\ The proof of Proposition \ref{Peucgen} in \cite{GoSiDec} is
constructive in the sense that starting with an $\al$--decomposition
\nn{al-iddec}, there is a simple recipe which describes how to
construct the operators $Q_{J,j}$ from \nn{iddecdual}, as required for the
dual $\al^u$--decomposition, and then vice versa.

2.\ The operators $Q_J$ in \nn{al-iddec} and 
$Q_{J,j}$ in \nn{iddecdual} are not given uniquely. Also we can see 
an obvious duality between $Q_{i,j}$, $Q_{j,i}$  and
$P_i$, $P_j$ in \nn{paireucgen}. This is discussed in 
\cite{GoSiDec} where notion of (dual) $\al$--decompositions is 
formulated in the language of Koszul complexes.
\end{remark}

From the practical point of view, given an operator $P: \cV \to \cV$
as in \nn{compP}, to apply Theorem \ref{winhg} one needs to show
whether $P$ is a (dual) $\al$--decomposition and also to determine
explicitly the corresponding operators $Q_J$ (or $Q_{J,j}$ in the dual
case). Also, one can ask which choice of $\al$ yields the most
suitable $\al$--decomposition.  Another strategy might be to
``regroup'' the operators $P_i$ (e.g.\ to consider the product
$P_iP_j$ as a single factor) and then to seek a better
$\al$--decomposition. In the case that the operator $P$ is polynomial
in other mutually commuting operators $\cD_j: \cV \to \cV$, there is a
category of decompositions which arise algebraically from the formula
for $P$. Within this category all these questions can all be solved in
a completely algorithmic way.  

\section{Operators polynomial in commuting endomorphisms and algebraic 
decompositions}
\label{multivar}

Writing  $\cV$ to denote a vector space over some field $\bF$,
suppose that $\cD_i:\cV\to \cV$, $i=1,\cdots ,k$, are non-trivial
linear endomorphisms that are mutually commuting:
$\cD_i\cD_j=\cD_j\cD_i$ for $i,j\in\{1,\cdots ,k\}$ . We obtain a
commutative algebra $\bF[\bD]$ consisting of those endomorphisms
$\cV \to \cV$ which may be given by expressions polynomial (with
coefficients in $\bF$) in the $\cD_i$.  We write $\bx =
(x_1,\ldots,x_k)$ for the multivariable indeterminate, and $\F[\bx]$
for the algebra of polynomials in the variables $x_1,\ldots,x_k$ over
the field $\F$. There is a unital algebra epimorphism from 
$\bF[\bx]$ onto $\bF[\bD]$ given by formally replacing each variable 
$x_i$, in a polynomial, with $\cD_i$. 

The simplest case is when $k=1$, that is operators polynomial in a
single operator $\cD$. We write $\bF[\cD]$ for the algebra of these.
Since any linear operator $\cD:\cV\to \cV$ is trivially self-commuting
there is no restriction on $\cD$. Thus this case is an important
specialisation with many applications.  In this setting we may quickly
find algebraic decompositions.  Let us
write $\bF[x]$ for polynomials in the single indeterminate $x$ and 
illustrate the idea with a very simple case. 
 Consider a polynomial $P[x]=(x+\lambda_0)(x+\lambda_1)
\cdots (x+\lambda_\ell)$ where for $i=0,\cdots ,\ell$, the
$\lambda_i\in \bF$ are are mutually distinct (i.e. $i\neq j$
$\Rightarrow$ $\lambda_i\neq \lambda_j $). Related to $P[x]$ are the
polynomials obtained by omitting a factor
$$ 
P^i[x]=\prod_{i\neq j=0}^\ell(x+\lambda_j).
$$ Then we associate to $P[x]$ the following decomposition of the unit
in $ \bF[x]$.
\begin{lemma}\label{pf0pol}
$$
1=\alpha_0P^0[x]+\alpha_1 P^1 [x]+\cdots +\alpha_\ell P^\ell [x]~,
$$
where 
$$
\alpha_i = \prod_{i\neq j =0}^{j=\ell} \frac{1}{\la_j-\la_i}~.
$$
\end{lemma}
\begin{proof} For $\ell=0$ we take the first display to mean $1=1$. 
For $\ell=1$ the Lemma states that 
$$
1=\frac{1}{\la_1-\la_0}(x+\la_1)+ \frac{1}{\la_0-\la_1}(x+\la_0)
$$ which is clearly true.  
Now assume that the result holds for all polynomials with $\ell-1$
factors. In particular associated to
$P^\ell[x]=\prod_{i=0}^{\ell-1}(x+\la_i)$ and
$P^0[x]=\prod_{i=1}^{\ell}(x+\la_i)$ we have the identities
\begin{equation}\label{one}
1= \beta_0 Q_0[x]+\cdots + \beta_{\ell-1}Q_{\ell-1}[x]
\quad \mbox{and}\quad 
1= \gamma_1 R_1[x]+\cdots + \gamma_{\ell}R_{\ell}[x], 
\end{equation}
respectively, where we have
$$ Q_i[x]= \prod_{i\neq j=0}^{\ell-1}(x+\lambda_j), \quad \beta_i
=\prod_{i\neq j =0}^{\ell-1} \frac{1}{\la_j-\la_i}, \quad \mbox{for}
\quad i=0,\cdots ,\ell-1
$$
and 
$$
\quad R_k[x]=
\prod_{k\neq j=1}^\ell(x+\lambda_j), \quad \gamma_k = \prod_{k\neq j 
=1}^{\ell} \frac{1}{\la_j-\la_k},  \quad \mbox{for} \quad k=1,\cdots ,\ell~.
$$ Now multiplying the first identity of \nn{one} by $(x+\la_\ell)$,
multiplying the the second identity of \nn{one} by $(x+\la_0)$ and then
taking the difference yields
$$ \la_\ell-\la_0= \beta_0 P^0[x]+(\beta_1-\gamma_1)P^1[x]+\cdots
+(\beta_{\ell-1}-\gamma_{\ell-1})P^{\ell-1}[x]-\gamma_\ell P^{\ell}[x].
$$ This establishes the result as $\beta_0/(
\la_\ell-\la_0)=\alpha_0$, $\gamma_\ell/(
\la_\ell-\la_0)=-\alpha_\ell$, while for $i=1,\cdots, \ell-1$ we have
$$
\frac{\beta_i-\gamma_i}{\la_\ell-\la_0}
=\big( \prod_{i\neq j=1}^{\ell-1} \frac{1}{\la_j-\la_i}\big)
\big( \frac{1}{\la_0-\la_i}- \frac{1}{\la_\ell-\la_i}\big)
\frac{1}{\la_\ell-\la_0} =\alpha_i~.
$$
\end{proof}
Thus we have the following.
\begin{proposition}\label{sim}
For $P[\cD]=(\cD+\lambda_0)(\cD+\lambda_1)
\cdots (\cD+\lambda_\ell)$ we have a decomposition given by 
$$
id_\cV= Q_0 P^0[\cD]+\cdots +Q_\ell P^\ell[\cD],
$$ where $P^i=\prod_{i\neq j=0}^\ell (\cD+\la_j)$ and $Q_i=\alpha_i$
for $i=0,\cdots ,\ell$.
\end{proposition}
Thus we may immediately apply Theorem \ref{winhg}, and in fact the
stronger variants for decompositions as in \cite{GoSiDec}, to reduce
homogeneous or inhomogeneous problems for $R$ to corresponding
problems of the form $(\cD +\la_i)u_i=f$. 
\begin{remark}
 However the point we wish to emphasise heavily is that we used no
information about the operator $\cD$ to obtain the decomposition in
Proposition \ref{sim}; $\cD$ can be any linear operator $\cD:\cV\to
\cV$ on any vector space $\cV$. Thus we will say that Proposition
\ref{sim} is an {\em algebraic decomposition} of $P[\cD]$. For
specific operators $\cD$ there may be other decompositions (or
$\al$-decompositions) that do use information about $\cD$. 

For operators polynomial in mutually commuting operators $\cD_i$ we
generically may obtain $\al$-decompositions that are algebraic in this
way; that is they arise, via the algebra epimorphism $\bF[\bx]\to
\bF[\bD]$, from a polynomial decomposition of the unit in
$\bF[\bx]$. These are universal $\al$-decompositions that are
independent of the details of the $\cD_i$, $i=1,\cdots, k$.
\end{remark}

Via the Euclidean algorithm, and related tools more powerful for these
purposes, we may easily generalise Lemma \ref{pf0pol} to obtain
decompositions for operators more interesting than $P[\cD]$ as in the
Proposition above. The case of a operators polynomial in a single
other operator is treated in some detail in \cite{GoSiDec} so let us
now turn our attention to some general features which appear more 
in the multivariable case $ k \geq 2$.

Given polynomials $P_0[\bx], P_1[\bx], \cdots ,P_{\ell}[\bx] \in \F[\bx]$
consider the product polynomial
\begin{equation} \label{Pmulti}
  P[\bx] = P_0[\bx] P_1[\bx] \cdots P_{\ell}[\bx].
\end{equation}
With $L=\{0,1,\cdots ,L\}$, we carry over, in an obvious way, the
labelling from Section \ref{setup} via elements
of the power set $2^L$; products of the polynomial $P_i[\bx]$ are
labelled by the corresponding subset of $L$. For example for
$J\subseteq L$, $P_J[\bx]$ means $\prod_{j\in J}P_j[\bx]$, while
$P^J[\bx]$ means $P_{L\setminus J}[\bx]$.

Considering the dual $\be$--decompositions, we need to verify that for each 
$ I \in \be$ we have
\begin{equation}\label{duale}
  1 \in \langle P_i[\bx]~:~ i\in I  \rangle
\end{equation}
where $\langle .. \rangle$ denotes the ideal in $\F[\bx]$ generated by
the enclosed polynomials.  It is useful to employ algebraic geometry
to shed light on this problem, in particular to use the ``algebra --
geometry dictionary'', see for example \cite[Chapter 4]{CLS}.  Let us
write $\cN(S[\bx]) := \{ \bx \in \F^k \mid S[\bx]=0 \}$ for the
algebraic variety determined by the polynomial $S[\bx]$.  The ideal
$\langle P_i[\bx] \rangle$ corresponds to the variety $\cN_i :=
\cN(P_i[\bx])$ and the previous display clearly requires $\bigcap_{i
\in I} \cN_i = \emptyset$.  In fact if $\F$ is algebraically closed then
the latter condition is equivalent to \nn{duale}. (This follows from
the Hilbert's Nullstellensatz, see \cite{CLS}.) Since generically
$\bigcap_{i \in I} \cN_i$ has codimension $|I|$, we conclude that (for
$\F$ algebraically closed) if $|I| \geq k+1$ then in the generic case
\nn{duale} will be satisfied.  

\vspace{2ex}
\noindent
\idx{(Dual) decompositions and $\al$-decompositions}
\vspace{1ex}

Aside from invertible $P$, the decompositions are the ``best
possible'' among all $\al$--decompositions (and similarly for the dual
versions).  However they require $2 \geq k+1$ in the generic case (we
need $|I|=2$ in \nn{duale}) which holds only for one variable
polynomials.  On the other hand there is always a chance that we
obtain a decomposition by a suitable ``regrouping'' of the
polynomials in \nn{Pmulti}. So we can proceed as follows.

Any polynomial $P[\bx] \in \F[\bx]$ can be decomposed into
irreducibles.  If we were to take $P_i[\bx]$ in \nn{Pmulti} 
as such irreducibles then \nn{Pmulti} would not be generally the 
decomposition in the multivariable case.  
To obtain the decomposition one can consider products $P_I[\bx]$
as single factors in \nn{Pmulti} for suitable $I \subseteq L$.
This reduces the number of factors (i.e.\ $\ell$);
to find an optimal (i.e.\ with $\ell$ maximal) version of this
procedure we use the following lemma.

\begin{lemma} \label{Pdecomp}
(i) Assume $P[\bx]$ has the form \nn{Pmulti} satisfying 
$1 \in \langle P_i[\bx], P_j[\bx]\rangle$ for all
$0 \leq i<j \leq \ell$  and
$P[\bx] = R_0[\bx] \cdots R_r[\bx]$ is a decomposition of $P[\bx]$
into irreducible polynomials $R_i[\bx]$. If
$1 \not\in \langle R_p[\bx], R_q[\bx] \rangle$
for some $0 \leq p,q \leq r$
then there exists $0 \leq i \leq \ell$ such that
$P_i[\bx] = R_p[\bx]R_q[\bx]P'_i[\bx]$ for a polynomial $P'_i[\bx]$.

(ii) Assume the polynomials $S_0[\bx],\ldots,S_s[\bx]$ and
$T_0[\bx],\ldots,T_t[\bx]$ satisfy $1 \in \langle S_i[\bx],T_j[\bx] \rangle$
for all $0 \leq i \leq s$ and $0 \leq j \leq t$ .
Then  $1 \in \langle S[\bx], T[\bx] \rangle$ where
$S[\bx] = S_0[\bx] \cdots S_s[\bx]$ and $T[\bx] = T_0[\bx] \cdots T_t[\bx]$.
\end{lemma}

\begin{proof}
(i) Assume the case $P_i[\bx] = R_p[\bx]P''_i[\bx]$ and
$P_j[\bx] = R_q[\bx]P''_j[\bx]$
for some $i \not= j$. Then $1 \in \langle P_i[\bx], P_j[\bx] \rangle$
implies $1 \in \langle R_p[\bx], R_q[\bx] \rangle$.

(ii) We use the induction with respect to $s+t$. Clearly the lemma holds for
$s+t=0$ so assume $s+t \geq 1$. Then e.g.\ $t \geq 1$ so by the inductive
hypothesis we get $1 \in \langle S[\bx], \tilde{T}[\bx] \rangle$
and $1 \in \langle S[\bx], T_t[\bx] \rangle$ where
$\tilde{T}[\bx] = T_0[\bx] \cdots T_{t-1}[\bx]$. This means
$$ 1 = a[\bx] S[\bx] + b[\bx] \tilde{T}[\bx]
   \quad \mbox{and} \quad
   1 = c[\bx] S[\bx] + d[\bx] T_t[\bx] $$
for some polynomials $a[\bx]$, $b[\bx]$, $c[\bx]$ and $d[\bx]$.
Now multiplying the right hand sides of these two equalities and using
$T[\bx] = \tilde{T}[\bx]T_t[\bx]$, the lemma follows.
\end{proof}

We will use this lemma as follows. We start with the decomposition of
$P[\bx] = R_0[\bx] \cdots R_r[\bx]$ into
irreducibles. Consider the graph with vertices $v_0,\ldots,v_r$,
and an edge $\{v_p,v_q\}$ for every $0 \leq p,q \leq r$ such that
$1 \not\in \langle R_p[\bx], R_q[\bx] \rangle$.
Denote the number of connected components by $\ell+1$ and the set of vertices
in the $i$th component by $G_i$, $0 \leq i \leq \ell$. We put
$$ P_i[\bx] := \prod_{v_u \in G_i} R_u[\bx], \quad i = 0,\ldots,\ell
$$ 
which yields the form \nn{Pmulti} of $P[\bx]$. This satisfies
$1 \in \langle P_i[\bx], P_j[\bx]\rangle$ for all 
$0 \leq i<j \leq \ell$ according to Lemma \ref{Pdecomp} (ii) and
thus \nn{Pmulti} is the decomposition.  Moreover,
it follows from the part (i) of the lemma that no form 
$P[\bx] = P'_0[\bx] \cdots P'_{\ell'}[\bx]$ with $\ell'>\ell$ can satisfy 
the condition $1 \in \langle P_i[\bx], P_j[\bx]\rangle$ for all
$0 \leq i<j \leq \ell'$. 
(The discussed graph has $\ell+1$
   connected components and we need to ``regroup'' the vertices
   (corresponding to irreducible components) into $\ell'+1$ groups
   corresponding to $\ell'+1$ polynomials $P'_i[\bx]$. If $\ell'>\ell$
   then there is a pair of irreducible polynomials $R_p[\bx],
   R_q[\bx]$ such that $1 \not\in \langle R_p[\bx], R_q[\bx] \rangle$
   which satisfy that $R_p[\bx]$ is a factor of $P'_i[\bx]$ and
   $R_p[\bx]$ is a factor of $P'_j[\bx]$ for some $0 \leq i < j \leq
   \ell'$. This is a contradiction with Lemma \ref{Pdecomp} (i).)

\begin{remark}
From the geometrical point of view, 
if \nn{Pmulti} is a decomposition then $\cN := \cN(P) = \bigcup \cN_i$
is the disjoint union. The previous paragraph describes
how to find such decomposition for the variety $\cN$ corresponding
to any $P[\bx]$ given by \nn{Pmulti}. 
Moreover, the obtained decomposition is minimal in the sense that
$\cN_i$ in $\cN = \bigcup \cN_i$ cannot be disjointly decomposed into 
smaller (nonzero) varieties.
\end{remark}

Generically, the (dual) decompositions are not available in the 
multivariable case. The ``optimal'' choice among all possible 
$\al$--decompositions (in the sense of \cite{GoSiDec}) is as follows.
The subsets $\be \subseteq 2^L$ are partially ordered by inclusion
(i.e.\ now we use the poset structure of $2^{2^L}$).  Given an
operator $P$ in the form \nn{Pmulti} consider the family $\Gamma$ of
systems $\be$ such that \nn{Pmulti} is a dual
$\be$--decomposition. Then $\Gamma$ has the greatest element $\be_P =
\bigcup_{\be \in \Gamma} \be$. Then an ``optimal'' choice for the
dual $\be$--decomposition of $P$ is $\be := \Min(\be_P)$.  (We want
to have in $\be$ to the smallest possible subsets of $L$. So if the
$P_i$s are not invertible then the case of a dual decomposition may be
regarded as the best we can do. With this philosophy we thus take
$\be_P$.  Then using Lemma \ref{MinMax} we take $\be := \Min(\be_P)$
as it is easier to work with a smaller number of subsets.)
Consequently, we obtain the optimal choice $\al := \Max((\be_P)^l)$
for the $\al$--decomposition of $P$.

\vspace{2ex}
\noindent
\idx{Algorithmic approach and the Gr\"obner basis}
\vspace{1ex}

Summarising, starting with $P[\bx]$, the problem of obtaining a factoring
\nn{Pmulti} which is the decomposition or a suitable 
$\al$--decomposition  boils down to testing the  condition
$1 \in \langle P_i[\bx] \mid i \in I \rangle$ 
for various subsets $I \subseteq L$.
This can be done using
Buchberger's algorithm
which computes a canonical basis (for a given ordering of monomials)
of the ideal $\langle P_i[\bx] \mid i \in I \rangle$, a so called
reduced Gr\"{o}bner basis \cite{CLS}.
If $1 \in \langle P_i[\bx] \mid i \in I \rangle$,
this basis has to be $\{1\}$.

In practice for reasonable examples this algorithm may be implemented in,
for example, Maple.  Actually, one can save some computation and
moreover obtain the explicit form of the operators 
$Q_J$ from \nn{al-iddec} or $Q_{J,j}$ from \nn{iddecdual} by 
using Buchberger's
algorithm without seeking the reduced basis.
Consider the ideal $I := \langle G \rangle \subseteq \F[\bx]$ (for a
set of polynomials $G$) such that $1 \in I$, and a Gr\"{o}bner basis
$G'$ of $I$. Then $\al \in G'$ for some scalar $\al \in \F$.  The
Buchberger's algorithm starts with $G$ and builds $G' \supseteq G$ by
adding various linear combinations (with coefficients in $\F[\bx]$) of
elements from $G$.  So when $\al$ is added (and the algorithm stops),
it has the required form which expresses $1$ as a linear 
combination of elements of $G$ (up to a scalar multiple $\al$).

\vspace{2ex}
\noindent
\idx{Example}
\vspace{1ex}

We shall demonstrate the previous observations on 
the operator 
\begin{align*}
  P[\cD_1,\cD_2] :=
  & \cD_1^{5}\cD_2 + \cD_1^4\cD_2^2 + 3\cD_1^4\cD_2 + \cD^4 
    + 3\cD_1^3\cD_2^{2} + 3\cD_1^{3}\cD_2 \\
  & + 2\cD_1^{3} + 3\cD_1^2\cD_2^2 + \cD_1^2\cD_2 + \cD_1\cD_2^2 - \cD_1.
\end{align*}
This correspond (after factoring) to the polynomial
\begin{equation} \label{Pexample}
  P[x,y] = (x+1)(xy+y+1)x(x^2+xy+x+y-1).
\end{equation}
For example, taking
  $\cD_1=\frac{\partial}{\partial x}$ and
  $\cD_2=\frac{\partial}{\partial y}$ the differential operator
  $P[\frac{\partial}{\partial x}, \frac{\partial}{\partial y}]:
  C^{\infty}(\R^2) \to C^{\infty}(\R^2)$ is a the sixth order
  differential operator.  We apply the previous observation and
  Theorem \ref{winhg} to reduce the corresponding differential
  equation to a lower order problem. In general, we start with the
  equation $P[\cD_1,\cD_2]u =f$ for a given $f \in \cV$.

First we shall find the optimal (in the sense as above) dual 
$\be$--decomposition, $\be \subseteq 2^{\{0,1,2,3\}}$ 
and/or whether we can obtain  the dual 
decomposition after an appropriate regrouping of the factors in
\nn{Pexample}. Many steps can be done directly in Maple. The factors
$$ P_0[x,y] = x+1, P_1[x,y] = xy+y+1, P_2[x,y] = x, 
   P_3[x,y] = x^2+xy+x+y-1 $$
are irreducible; this can be verified by the command {\tt IsPrime}.
Using {\tt gbasis} we see that 
\begin{eqnarray} \label{ideals}
\begin{split}
  &1 \in \langle P_0[x,y], P_i[x,y] \rangle, \quad i \in \{1,2,3\}, \\
  &1 \not\in \langle P_i[x,y], P_j[x,y] \rangle, \quad i,j \in \{1,2,3\} 
  \quad \mbox{and}\\
  &1 \in \langle P_1[x,y], P_2[x,y], P_3[x,y] \rangle.
\end{split}
\end{eqnarray}
From the first two lines in \nn{ideals} and using the observation around
Lemma \ref{Pdecomp} we conclude that $P[x,y]=P_0[x,y]\tilde{P}[x,y]$
is the decomposition for $\tilde{P}[x,y] = (xy+y+1)x(x^2+xy+x+y-1)$.
Following the first line in \nn{ideals}, one easily computes 
\begin{eqnarray*} 
\begin{split}
  &1 = -yP_0[x,y] + P_1[x,y], \quad 1 = -(x-1)P_0[x,y] + xP_2[x,y], \\
  &1 = (x+y)P_0[x,y] - P_3[x,y]
\end{split}
\end{eqnarray*}
and multiplying these three relations we obtain 
$$ 1= \tilde{Q}[x,y] P_0[x,y] + 
      Q_0[x,y]\underbrace{P_1[x,y]P_2[x,y]P_3[x,y]}_{\tilde{P}[x,y]} $$
together with explicit form of the projectors $\tilde{Q}$ and $Q_0$.
Passing to the corresponding operators on the space $\cV$, this is
the decomposition \nn{al-iddec} of 
$P[\cD_1,\cD_2] = P_0[\cD_1,\cD_2]\,\tilde{P}[\cD_1,\cD_2]$.
Now using Theorem \ref{winhg} (and Remark \ref{dechg}) we see that 
every solution $u \in \cV$ of $P[\cD_1,\cD_2]u=f$ can be uniquely 
expressed as
$$ u = \tilde{Q}[\cD_1,\cD_2] u_0 + Q_0[\cD_1,\cD_2]\tilde{u} $$
where $u_0$ and $\tilde{u}$ satisfy
$P_0[\cD_1,\cD_2]u_0=f$ and $\tilde{P}[\cD_1,\cD_2]\tilde{u}=f$.
So we have reduced the original problem $P[\cD_1,\cD_2]u=f$ to the 
system of latter two equations. 

Using the last line in \nn{ideals}, we can apply Theorem \ref{winhg} 
to the equation $\tilde{P}[\cD_1,\cD_2]\tilde{u}=f$. It is easy 
to compute the corresponding dual $\be$--decomposition for
the operator $\tilde{P}[\cD_1,\cD_2] = 
P_1[\cD_1,\cD_2]P_2[\cD_1,\cD_2]P_3[\cD_1,\cD_2]$, 
$\be = \{\{1,2,3\}\}$; on the polynomial level we obtain 
$$ 1= \frac{1}{2}P_1[x,y] + \frac{1}{2}(x+1)P_2[x,y] 
      - \frac{1}{2}P_3[x,y]. $$
This is actually also the $\al$--decomposition, 
$\al = \{\{1,2\},\{2,3\},\{1,3\}\}$ as 
$P_1[\cD_1,\cD_2] = P^{\{2,3\}}[\cD_1,\cD_2]$ etc.\ according to 
the notation in \nn{al-iddec}. Now applying Theorem \ref{winhg} we 
obtain that every solution $\tilde{u} \in \cV$ of 
$\tilde{P}[\cD_1,\cD_2]\tilde{u}=f$ has the form 
$$ \tilde{u} = \frac{1}{2}\tilde{u}_1 + \frac{1}{2}(\cD_1+1)\tilde{u}_2
   - \frac{1}{2}\tilde{u}_3 $$
where $\tilde{u}_1$, $\tilde{u}_2$ and $\tilde{u}_3$ satisfy
the equations 
$P_2[\cD_1,\cD_2]\,P_3[\cD_1,\cD_2]\,\tilde{u}_1 =f$, 
$P_1[\cD_1,\cD_2]\,P_3[\cD_1,\cD_2]\,\tilde{u}_2 =f$ and
$P_1[\cD_1,\cD_2]\,P_2[\cD_1,\cD_2]\,\tilde{u}_3 =f$. Note that
such expression for $\tilde{u}$ is not generally unique.

Summarising, we have reduced the original equation
$P[\cD_1,\cD_2]u=f$, see \nn{Pexample}, for $\cD_1,\cD_2 \in \End(\cV)$
to the system of four equations 
\begin{eqnarray} \label{Presult}
\begin{split}
  &&P_i[\cD_1,\cD_2]&P_j[\cD_1,\cD_2]\tilde{u}_k =f 
    \ \mbox{where}\ \{i,j,k\} = \{1,2,3\},\ i<j, \\
  &&P_0[\cD_1,\cD_2]&u_0=f.
\end{split}
\end{eqnarray}
If we put $\cD_1 := \frac{\partial}{\partial x}$ and $\cD_2 :=
\frac{\partial}{\partial y}$, the original problem has order $6$ and
the resulting system \nn{Presult} has order $3$.

\section{Systems of polynomial equations}
\label{systems}

The notion algebraic decompositions from \cite{GoSiDec} summarised in 
Section \ref{setup} can be applied also to systems of equations of the 
form \nn{compP} with commuting $P_i$. Here we describe one possible 
type of such a system to demonstrate power of this machinery.
We will consider only the (true) decompositions.

Let us consider a $k$--tuple of commuting linear operators 
$P^{(i)} \in \End{\cV}$ and corresponding equations 
\begin{equation} \label{system}
  P^{(i)} u = f^i, \quad [P^{(i)},P^{(j)}]=0,\ f^i \in \cV,\ 
  1 \leq i,j \leq k.
\end{equation}
The necessary (i.e.\ integrability) condition for existence of 
a solution $u$ is obviously
\begin{equation} \label{integr}
  P^{(i)} f^j = P^{(j)} f^i, \quad 1 \leq i<j \leq k.
\end{equation}
If, for some $i$, $P^{(i)}$ is of the form \nn{compP} satisfying
\nn{al-iddec}, 
we can replace $P^{(i)} u = f^{(i)}$ with several 
simpler equations using Theorem \ref{winhg}. But even if this is not the case,
one can obtain a decomposition using an algebraic relation between the
$P^{(i)}$s. 

\vspace{1ex}

Let us consider the special case where just one equation from 
\nn{system} is of the form
\nn{compP} and we do not decompose the remaining ones, i.e.\ we have the system
\begin{eqnarray} \label{1compP}
\begin{split}
  &Pu = P_0 P_1 \ldots P_\ell u = f,\ 
   [P_i,P_j]=0, 1 \leq i,j \leq \ell, \\
  &R^{(j)} u = g^j,\  
   [P_i,R^{(j)}]=0, 1 \leq j \leq k, 1 \leq i \leq \ell
\end{split}
\end{eqnarray}
where $P_i$ and $R_j$ satisfy the identity 
\begin{eqnarray} \label{1iddec}
\begin{split}
  &id_V = Q_0 P^0 + \ldots + Q_\ell P^\ell + 
         S_1 R^{(1)} + \ldots + S_k R^{(k)}, \\
  &[Q_i,P_j] = [S_p,R_q] =0,\
   1 \leq i,j \leq \ell,\ 1 \leq p,q \leq k, \\
  &[Q_i,R^{(p)}] = [S_p,P_i] =0,\
   1 \leq i \leq \ell,\ 1 \leq p \leq k
\end{split}
\end{eqnarray}
where $P^i:=\Pi_{i\neq j=0}^{j=\ell} P_i, i=0,\cdots ,\ell$.
(Note in \nn{1compP} we require not only commutativity of the left 
hand sides of the equations in the systems as in \nn{system} 
but also commutativity of the factors $P_i$ and the left hand sides
$R^{(j)}$.)
The condition \nn{integr} then becomes
\begin{equation} \label{1integr}
  R^{(j)} f = P g^j, \ R^{(j)} g^i = R^{(i)} g^j \quad 
  \mbox{for all} \ 1 \leq i,j \leq k.
\end{equation}

\begin{proposition}[Dual decomposition] \label{1Peucgen}
Let us consider the system \nn{1compP}. Then 
\nn{1iddec} is equivalent to
\begin{equation} \label{1paireucgen}
  \id_{\cV} = Q_{i,j} P_i + Q_{j,i} P_j + 
              Q_{i,j}^1 R^{(1)} + \ldots + Q_{i,j}^k R^{(k)}
\end{equation}
where $Q_{i.j}, Q_{i.j}^p \in \End (\cV)$
satisfy $[Q_{i,j},P_s] = [Q_{i,j}R^{(p)}] = [Q_{i,j}^p,R^{(q)}] = 
[Q_{i,j}^p,P_s] =0$
where $0 \leq i,j,s \leq \ell$, $i \not= j$ and $1 \leq p,q \leq k$.
\end{proposition}

\begin{proof}
This is just a straightforward modification of the proof of Proposition
\ref{Peucgen}.
\end{proof}

In this setting, we obtain an analogue of Theorem \ref{winhg} for the 
special case of decompositions. In this
Theorem, we replaced the operator $P$ given by \nn{compP} satisfying 
\nn{al-iddec} with
the system \nn{wmultiprobg} of $\ell+1$ simpler equations. Here we replace
the system \nn{1compP} with a ``system of simpler systems'' as follows.

\begin{theorem} \label{1inhg}
Let $\cV$ be a vector space over a field $\bF$
and consider $P, R^{(j)}: \cV \to \cV$ as in \nn{1compP} with
the factorisation giving the decomposition \nn{1iddec}. 
Here and below we assume the range $1 \leq j \leq k$.
Let us fix $f, g^j \in \cV$.  There is a 1-1 relationship
between solutions $u \in \cV$ of \nn{1compP} and solutions 
$(u_0,\cdots,u_\ell)\in\oplus^{\ell+1}\cV$ 
of the problem
\begin{eqnarray} \label{1multiprobg}
\begin{split}
  P_0 u_0=f, R^{(1)} u_0 &= P^0 g^1, \cdots, R^{(k)} u_0 = P^0 g^k \\   
  & \vdots \\
  P_\ell u_\ell=f, R^{(1)} u_\ell &= P^\ell g^1, \cdots, 
    R^{(k)} u_\ell = P^\ell g^k.
\end{split}
\end{eqnarray}

Writing $\cV_P^{f,\bg}$ for the solution space of \nn{1compP} and (for
$i=0,\cdots ,\ell$) $\cV^{f,\bg}_{i}$ for the solution space of
the system corresponding to the $i$th line in \nn{1multiprobg}, where
$\bg = (g^1,\ldots,g^k)^T$.
The map 
$F:\cV^{f,\bg}_P\to \mbox{\Large\bf $\times$}_{i=0}^\ell\cV^{f,\bg}_{i} $ 
is given by
$$ u\mapsto (P^0 u, \cdots ,P^\ell u ) ~,
$$
with inverse 
$B: \mbox{\Large\bf $\times$}_{i=0}^\ell\cV^{f,\bg}_{i}\to \cV_P^{f,\bg}$ 
given by
$$
(u_0,\cdots, u_\ell) \mapsto 
\sum_{i=0}^{i=\ell} Q_i u_i + \sum_{j=1}^{j=k} S_j g^j~.
$$
On $\cV$ we have $B\circ F=id_{\cV_P^{f,\bg}}$, while on the affine space
$\mbox{\Large\bf $\times$}_{i=0}^\ell\cV^{f,\bg}_{i}$ we have
$F\circ B =\id_{\mbox{\large\bf $\times$}_{i=0}^\ell\cV^{f,\bg}_{i}}$.
\end{theorem}

\begin{proof} Suppose $u$ is a solution of \nn{1compP}. 
Then $P_iP^i u=Pu=f$ and also 
$R^{(j)} P^i u = P^i(R^{(j)}u) = P^i g^j$. Hence $Fu$ 
is a solution of \nn{1multiprobg}. 
For the converse suppose that $(u_0,\cdots,u_\ell)$ is a solution of 
\nn{1multiprobg} and write 
$u:=  \sum_{i=0}^{i=\ell} Q_i u_i + \sum_{j=1}^{j=k} S_j g^j$.
Then
$$
\begin{aligned}
  Pu= \sum_{i=0}^{i=\ell}PQ_i u_i + \sum_{j=1}^{j=k} PS_j g^j =
  \sum_{i=0}^{i=\ell} Q_i P^i (P_i u_i) + \sum_{j=1}^{j=k} S_j R^{(j)}f =f
\end{aligned}        
$$
where we have used $Pg^j = R^{(j)}f$ from \nn{1integr} and then \nn{1iddec}. 
Further
$$
\begin{aligned}
  R^{(j)}u 
  &= \Bigl( \sum_{i=0}^{i=\ell} R^{(j)} Q_i u_i \Bigr) 
     + R^{(j)} S_j g^j 
     + \Bigl( \sum_{j\not= p=1}^{p=k} R^{(j)} S_p g^p \Bigr) \\
  &= \Bigl( \sum_{i=0}^{i=\ell} Q_i P^i g^j  \Bigr)
     + R^{(j)} S_j g^j
     + \Bigl( \sum_{j\not= p=1}^{p=k} S_p  R^{(p)} g^j \Bigr) = g^j
\end{aligned}
$$
where we have used $R^{(j)} u_i = P^ig^j$ from \nn{1multiprobg}
and $R^{(j)} g^p = R^{(p)} g^j$ from \nn{1integr} in the middle equality
and \nn{1iddec} in the last one.

It remains to show that $F$ and $B$ are inverses. Clearly 
$u \in \cV_P^{f,\bg}$ 
satisfies 
$$ (B \circ F) u  = 
   \sum_{i=0}^{i=\ell} Q_iP^i u + \sum_{j=1}^{j=k}S_j g^j =
   \sum_{i=0}^{i=\ell} Q_iP^i u + \sum_{j=1}^{j=k}S_j R^{(j)} u = u$$
since $g^j = R^{(j)} u$ according to \nn{1compP}. Thus
we obtain $B\circ F=id_{\cV_P^{f,\bg}}$. To compute the opposite direction we 
need the $r^{\rm th}$ component of $FB(u_0,\ldots,u_\ell)$ 
for $(u_0,\ldots,u_\ell) \in \times_{i=0}^\ell\cV^{f,\bg}_{i}$.
This is equal to 
$$ P^r \sum_{i=0}^{i=\ell} Q_i u_i + P^r \sum_{j=1}^{j=k} S_j g^j = 
   \sum_{r \not= i=0}^{i=\ell} Q_i P^r u_i + Q_r P^r u_r
   + \sum_{j=1}^{j=k} S_j R^{(j)} u_r = u_r. $$
since $P^r u_i = P^i u_r$ for $r \not= i$ (the consistency condition given
by $P_iu_i =f$ for $i =0,\ldots,\ell$) and $P^r g^j = R^{(j)} u_r$ from
\nn{1multiprobg}. Hence $F \circ B = id_{\times_{i=0}^\ell\cV^{f,\bg}_{i}}$.
\end{proof}

\vspace{1ex}

\section{Higher symmetries of operators}\label{symm}

For a vector space $\cV$ and a linear operator $P:\cV\to \cV$, let is
say that a linear operator $S:\cV\to \cV$ is a {\em formal symmetry}
  of $P$ if $PS=S'P$, for some other linear operator
$S':\cV\to \cV$. Note that $S:\cN(P)\to \cN(P)$. In \cite{GoSiDec} we
called operators with the latter property ``weak symmetries'' 
and discussed the
structure of the algebra of these in relation to symmetries and
related maps for the component operators $P_i$.  We show here that
although formal symmetries are defined rather differently similar
results hold using our general tools as discussed above and in
\cite{GoSiDec}.  For the case of $P$ a differential operator the
formal symmetries agree with the ``higher symmetries'' considered in
\cite{MikEsym} and we thank Mike Eastwood for asking whether the ideas
from \cite{GoSiDec} might be adapted to deal directly with what we are
here calling formal symmetries.

Consider the case of an operator $P=P_0P_1\cdots P_\ell$
with a decomposition
\begin{equation}\label{decs}
id_\cV=Q_0P^0+\cdots +Q_\ell P^\ell,
\end{equation}
i.e.\ \nn{al-iddec} with $\al := \{ J \subseteq L \mid |J|=1\}$. 
Then as commented in Remark \ref{dechg}, upon restriction to
$\cN(P)$, the operators
$$
Pr_i:=Q_iP^i, \quad i=0,1,\cdots ,\ell
$$ are projections onto $\cN(P_i)$. This was the critical object used to 
discuss  weak symmetries and their decompositions in \cite{GoSiDec}.
Here we see that it plays a similar  for formal
symmetries.

Now note that if $S$ is a formal symmetry of $P$ then $Pr_i S Pr_i$ is
a formal symmetry of $P_i$. More generally, using the assumed
commutativity as in \nn{al-iddec}, we have
$$
P_i (Pr_i S Pr_j)= Q_i(P S) Pr_j=Pr_i (S'P)Pr_j=(Q_i S'Pr_jP^j)P_j  
$$ so $S_{ij}:=Pr_i S Pr_j $ linearly maps $\cN(P_j)\to \cN(P_i)$. 
(In fact $P^i S$ would suffice (see the remark below), 
we use $Pr_i S Pr_j$ for the link
with \cite{GoSiDec}.) 
But we may view the property $ P_iS_{ij}=
S'_{ij}P_j$ (with $S'_{ij}$ any linear endomorphism of $\cV$) as a
generalisation of the idea of a formal symmetry.  If we have such a
{\em generalised formal symmetry} $S_{ij}$ for all pairs $i,j\in
\{0,1,\cdots ,\ell\}$ then note that for each pair $i,j$ we have
$$
P  S_{ij}Pr_j=P^i P_iS_{ij}Pr_j=P^iS'_{ij}P_jPr_j=(P^iS'_{ij}Q_j)P; 
$$ $S_{ij}Pr_j$ (and hence also $Pr_iS_{ij}Pr_j$) is a formal symmetry
of $P$.  Thus the decomposition of the identity \nn{decs} allows us to
understand formal symmetries of $P$ in terms of the generalised formal
symmetries of the component operators $P_i$, $i=0,1,\cdots, \ell$.

\begin{remark}
  This result for formal symmetries follows the Theorem 4.1 in
  \cite{GoSiDec} where weak symmetries are treated. The decomposition
  of the identity \nn{decs} plays the crucial role in this theorem.
  Since,upon restriction to $\cN(P)$, the $Pr_i$ are projections, the
  formulae above have a straightforward conceptual interpretation.
  However there is, in fact, an even simpler relationship between
  formal symmetries $S$ of $P$ and generalised formal symmetries
  $S_{ij}$, $i,j \in \{0,\ldots,\ell\}$.  We simply put $S_{ij} :=
  P^iS|_{\cN(P_j)}$ for a formal symmetry $S$ and $S := S_{ij}P^j$ for
  a generalised formal symmetry $S_{ij}$.
\end{remark}

Using a factorisation from \cite{GoEinst}, this observation enables a
treatment of the higher symmetries of the e.g.\ the conformal Laplacian
operators of \cite{GJMS} on conformally Einstein manifolds. In
particular an alternative approach to the higher symmetries of the
Paneitz operator which is alternative to that in \cite{EL}. (In fact
in \cite{EL} they consider only the square of the Laplacian on
Euclidean space but by conformal invariance this may alternatively
treated via the Paneitz operator on the sphere.)  This will be taken
up elsewhere.

\end{document}